\documentclass[11pt, a4paper]{article}
\usepackage{amsmath, amsfonts}
\usepackage[usenames]{color}
\usepackage[latin1]{inputenc}
\usepackage[table, dvipsnames]{xcolor}
\usepackage{array}
\newcolumntype{P}[1]{>{\centering\arraybackslash}p{#1}}
\newcolumntype{M}[1]{>{\centering\arraybackslash}m{#1}}
\usepackage{multirow}
\usepackage[all]{xy}
\usepackage{mathtools}
\usepackage{mathrsfs}
\usepackage{hyperref} 
\hypersetup{bookmarks = true, 
                  colorlinks = true, 
                  linkcolor = blue, 
                  citecolor = blue, 
                  }
\newcommand{\xdownarrow}[1]{%
  {\left\downarrow\vbox to #1{}\right.\kern-\nulldelimiterspace}
}

\newcommand{\xuparrow}[1]{%
  {\left\uparrow\vbox to #1{}\right.\kern-\nulldelimiterspace}
}

\usepackage{graphicx}
\usepackage{amssymb}
\usepackage{amsmath}

\newtheorem{theorem}{Theorem}[section]

\newtheorem{conjecture}[theorem]{Conjecture}
\newtheorem{corollary}[theorem]{Corollary}

\newtheorem{definition}[theorem]{Definition}
\newtheorem{example}[theorem]{Example}

\newtheorem{lemma}[theorem]{Lemma}

\newtheorem{proposition}[theorem]{Proposition}

\newenvironment{proof}[1][Proof]{\noindent\textbf{#1.} }{\ \rule{0.5em}{0.5em}}

\title{\textbf{On invariants of a map germ from n-space to 2n-space}}
\author{ \ \ \ \\{Nuño-Ballesteros, J.J.  ,  $ \ \ $ Silva, O.N. $ \ \ $ and $ \ \ $ Tomazella, J.N.}}
\date{}

\begin{document}

\maketitle

\begin{abstract}
We consider $\mathcal{A}$-finite map germs $f$ from $(\mathbb{C}^n,0)$ to $(\mathbb{C}^{2n},0)$. First, we show that the number of double points that appears in a stabilization of $f$, denoted by $d(f)$, can be calculated as the length of the local ring of the double point set $D^2(f)$ of $f$, given by the Mond's ideal. In the case where $n\leq 3$ and $f$ is quasihomogeneous, we also present a formula to calculate $d(f)$ in terms of the weights and degrees of $f$. Finally, we consider an unfolding $F(x,t) = (f_t(x),t)$ of $f$ and we find a set of invariants whose constancy in the family $f_t$ is equivalent to the Whitney equisingularity of $F$. As an application, we present a formula to calculate the Euler obstruction of the image of $f$. 
\end{abstract}

\section{Introduction}

$ \ \ \ \ $ Our purpose in this work is to study $\mathcal{A}$-finite (or equivalently, finitely determined) map germs from $(\mathbb{C}^n,0)$ to $(\mathbb{C}^{2n},0)$. Many authors have studied this type of maps from several different points of view, such as maximal deformations and theory of classification, as in the search for topological and analytical invariants. In the literature we can cite for instance the work of Gaffney in \cite{gaffney2}, Klotz, Pop and Rieger in \cite{klotz-pop-rieger} and more recently Rodrigues Hernandes and Ruas in \cite{ruaselenice2} and \cite{ruaselenice}. 

It is well known that if $f:(\mathbb{C}^n,0)\rightarrow (\mathbb{C}^{2n},0)$ is $\mathcal{A}$-finite then the number of transverse double points which appears in a stabilization of $f$, denoted by $d(f)$, is finite and it is an analytic invariant of $f$. In this work, we are interested in studying a way of calculating the invariant $d(f)$ in terms of $f$ itself.  There are some ways already known in the literature to calculate this invariant. For instance, suppose that $f:(\mathbb{C}^n,0)\rightarrow (\mathbb{C}^{2n},0)$ is 
a map germ of corank $1$. In this case, after a suitable change of coordinates, we can write $f$ as 

\begin{equation}\label{eq8}
f(\textbf{x},y)=(x_1,\cdots,x_{n-1},f_n(\textbf{x},y),\cdots,f_{2n}(\textbf{x},y)) 
\end{equation}

\noindent where $\textbf{x}=(x_1,\cdots,x_{n-1})$. So, adding a new variable $y^{'}$, $d(f)$ can be calculated as

\begin{equation}\label{eq7}
d(f)=\dfrac{1}{2} \displaystyle \left( dim_{\mathbb{C}}\dfrac{\mathcal{O}_{n+1}}{ \left\langle (f_n(\textbf{x},y)-f_n(\textbf{x},y^{'}))/(y-y^{'}), \cdots, (f_{2n}(\textbf{x},y)-f_{2n}(\textbf{x},y^{'}))/(y-y^{'})\right\rangle } \right)
\end{equation}

\noindent since in this case, the Mond's ideal of the double point set of $f$ is a complete intersection (see \cite{mond89}).

For $f(\textbf{x})=(f_1(\textbf{x}),\cdots,f_{2n}(\textbf{x})): (\mathbb{C}^n,0)\rightarrow (\mathbb{C}^{2n},0)$ of any corank set $\textbf{x}^{'}=(x_1^{'},\cdots,x_n^{'})$. There is a double-point formula

\begin{equation}\label{eq10}
\epsilon(f)=2d(f)=dim_{\mathbb{C}}\dfrac{\langle x_1-x_1^{'},\cdots,x_n-x_n^{'}  \rangle }{\langle f_1(\textbf{x})-f_1(\textbf{x}^{'}),\cdots,f_{2n}(\textbf{x})-f_{2n}(\textbf{x}^{'}) \rangle }
\end{equation}

\noindent given by Artin and Nagata in \cite[p. 320]{artinnagata}.

We also have the work of Gaffney (see \cite[Cor. 3.3]{gaffney2}). In his result, he showed how one can calculate $d(f)$ in terms of an appropriated Segre number of the double point ideal of $f$ and the number of cross-caps that appears in a stabilization of $\pi \circ f$, where $\pi:(\mathbb{C}^{2n},0)\rightarrow (\mathbb{C}^{2n-1},0)$ is a generic linear projection from $(\mathbb{C}^{2n},0)$ to $(\mathbb{C}^{2n-1},0)$.

In this work, following \cite{mond87} we present the double point ideal $I^2(f)$ of a finite map germ $f:(\mathbb{C}^n,0)\rightarrow (\mathbb{C}^p,0)$, $n\leq p$, which gives an appropriated analytic structure for the double point set of $f$. Using this analytic structure, we show (see Proposition \ref{mainresult1}) that if $f:(\mathbb{C}^n,0)\rightarrow (\mathbb{C}^{2n},0)$ is $\mathcal{A}$-finite of any corank, then $d(f)$ is given as the codimension of the double point ideal $I^{2}(f)$ in $\mathcal{O}_{2n}$.

Our second result is about quasihomogeneous map germs. The case in f has corank $1$ is well known. More precisely, suppose that $f$ has corank $1$ and it is written as in (\ref{eq8}) and it is quasihomogeneous of weights $w_1,\cdots,w_n$ and weighted
degrees $d_1,\cdots d_{2n}$. Applying an appropriated version of Bézout's theorem in (\ref{eq7}) with $w_n$ as the weight of the additional variable $y^{'}$, it is well known that $d(f)$ can be calculated as

\begin{equation}\label{eq9}
d(f)=\dfrac{\prod_{i=n}^{2n}(d_i-w_n)}{2w_n^2\prod_{j=1}^{n-1}w_j}.
\end{equation}

When the corank of $f$ is greater than $1$, $D^2(f)$ is no longer a complete intersection. Hence, in this case we cannot apply directly Bézout's theorem to find a formula. However, using another technique, when $n=2$ or $3$ we extend formula (\ref{eq9}) for any corank (see Proposition \ref{quasihomformula}) and we present a conjecture for the case $n>3$. 

We also present another way to calculate $d(f)$ in terms of the $\epsilon$-invariant for isolated non-normal singularities introduced by Greuel in \cite[Sec. 5]{greuel2}. Our presentation of this result (Proposition \ref{mainresult4}) actually can be seen as another version of Artin-Nagata formula, with a point of view of isolated non-normal singularities.

We are also interested in studying invariants that control the Whitney equisingularity of a deformation of $f$. We consider an unfolding $F:(\mathbb{C}^n \times \mathbb{C},0)\rightarrow (\mathbb{C}^{2n}\times \mathbb{C},0)$, $F(\textbf{x},t)=(f_t(\textbf{x}),t)$, and we assume that it is origin preserving, which means that $f_t(0)=0$ for any $t$. 
Gaffney in \cite[Sec. 7]{gaffney} showed that the Whitney equisingularity, hence the topological triviality, of a 1-parameter family of map germs from $(\mathbb{C}^n,0)$ to $(\mathbb{C}^p,0)$ is controlled by the multiplicities of the local polar varieties and the zero stable invariants of all the stable types which appear in the source and in the target. In the case where $f$ has corank $1$, Jorge Pérez and Saia give presented in \cite[Th. 4.3]{victorsaia} a reduction in the number of invariants that control Whitney equisingularity. In the case where $p=2n$, we reduced the number of invariants needed to control Whitney equisingularity of $F$ in the main result of \cite[Sec. 7]{gaffney}, with no restriction on the corank of $f$. 

Finally, as an application of some of our results, we show (Corollary \ref{mainresult5}) that the Euler obstruction (introduced by R. MacPherson in \cite{macpherson}) of $f(\mathbb{C}^n)$, which we denote by $Eu_0(f(\mathbb{C}^n))$, can be calculated in terms of the Milnor number of a single hypersurface in the source of $f$. More precisely, we show that 

\begin{center}
$Eu_0(f(\mathbb{C}^n))=1-(-1)^n\mu(f^{-1}(H \cap f(\mathbb{C}^n))$, 
\end{center}

\noindent where $H$ is a generic hyperplane in $\mathbb{C}^{2n}$ and $\mu(X)$ denotes the Milnor number of $X$ at $0$. As a consequence, we conclude that the local Euler obstruction of $f_t(\mathbb{C}^n)$ is an invariant for the Whitney equisingularity of $F$. 



\section{Preliminaries}\label{sec1}

$ \ \ \ \ $ Throughout this paper, $\textbf{x}=(x_1,\cdots,x_n)$ and $\textbf{X}=(X_1,\cdots,X_{2n})$ will be used to denote systems of coordinates in $\mathbb{C}^n$ and $\mathbb{C}^{2n}$, respectively. Also, $\mathbb{C} \lbrace x_1,\cdots,x_n \rbrace \simeq \mathcal{O}_n$ denotes the local ring of convergent power series in $n$ variables. The letters $U,V$ and $T$ are used to denote open neighborhoods of $0$ in $\mathbb{C}^n$, $\mathbb{C}^{2n}$ and $\mathbb{C}$, respectively. We also use the standard notation of singularity theory as the reader can find in Wall's survey paper \cite{wall} (see also \cite{juanjomond}).

We recall the notion of $\mathcal{A}$-finite determinacy, where $\mathcal{A}$ is the group of coordinates change in the source and in the target, as defined in \cite{wall}. 

\begin{definition}\label{def a equi}(a) Two map germs $f,g:(\mathbb{C}^n,0)\rightarrow (\mathbb{C}^{p},0)$ are $\mathcal{A}$-equivalent, denoted by $g\sim_{\mathcal{A}}f$, if there exist germs of diffeomorphisms $\eta:(\mathbb{C}^n,0)\rightarrow (\mathbb{C}^n,0)$ and $\xi:(\mathbb{C}^{p},0)\rightarrow (\mathbb{C}^{p},0)$, such that $g=\xi \circ f \circ \eta$.\\

\noindent (b) A map germ $f:(\mathbb{C}^n,0) \rightarrow (\mathbb{C}^{p},0)$ is finitely determined (with respect to the group $\mathcal{A}$) if there exists a positive integer $k$ such that for any $g$ with $k$-jets satisfying $j^kg(0)=j^kf(0)$ we have $g \sim_{\mathcal{A}}f$.
\end{definition}

Consider a finite map germ $f:(\mathbb{C}^n,0)\rightarrow (\mathbb{C}^{2n},0)$. By Mather-Gaffney criterion (\rm\cite[Th. 2.1]{wall}), $f$ is finitely determined (equivalently $\mathcal{A}$-finite) if and only if there is a finite representative $f:U \rightarrow V$ such that $f^{-1}(0)=\lbrace 0 \rbrace$ and the restriction $f:U \setminus \lbrace 0 \rbrace \rightarrow V \setminus \lbrace 0 \rbrace$ is stable. This means that the only singularities of $f$ on $U \setminus \lbrace 0 \rbrace$ are isolated transverse double points. By shrinking $U$ if necessary, we can assume that there are no transverse double points in $U$. Hence, $f$ is $\mathcal{A}-$finite if and only if the image of $f$ has isolated singularity. Also, it is not hard to see that $f$ is $\mathcal{A}$-finite if it is an injective immersion outside the origin. 

A $1$-parameter unfolding of $f$ is a map germ $F:(\mathbb{C}^n \times \mathbb{C},0)\rightarrow(\mathbb{C}^{2n} \times \mathbb{C},0)$ of the form $F(x,t) = (f_t(x),t)$ such that $f_0 = f$. We say that an unfolding $F$ is a stabilization of $f$ if there is a representative $F : F^{-1}(V \times T) \rightarrow V \times T$, where $V$ and $T$ are open neighborhoods of $0$ in $\mathbb{C}^{2n}$ and $\mathbb{C}$ respectively, such that $f_t : U_t \rightarrow V$ is stable for any $t \in U_t \setminus 0$, where $U_t:=F^{-1}(V \times \lbrace t \rbrace)$. Since we are in the nice dimensions of Mather (\rm\cite[p. 208]{mather}), we can take a stabilization of a finitely determined map germ $f:(\mathbb{C}^n,0)\rightarrow (\mathbb{C}^{2n},0)$ (see \cite{marar93} for details of a construction of a stabilization for $f$). It is well known that the number of transverse double points of $f_t$, with $t\neq 0$, is independent of $t$ and of the particular choice of the stabilization. Then, we define

\begin{center}
$d(f):=$ the number of transverse double points of $f_t$, $ \ \ $, with $t\neq 0$.
\end{center}

\begin{example} Consider the map germ $f:(\mathbb{C},0)\rightarrow (\mathbb{C}^2,0)$ defined by $f(u)=(u^2,u^3)$. A stabilization of $f$ is given by $F=(f_t(u),t)=(u^2,u^3+tu,t)$ is a stabilization of $f$. In this case $d(f)=1$ which is illustrated in Figure \rm\ref{figura1}.

\begin{figure}[h]
\centering
\includegraphics[scale=0.4]{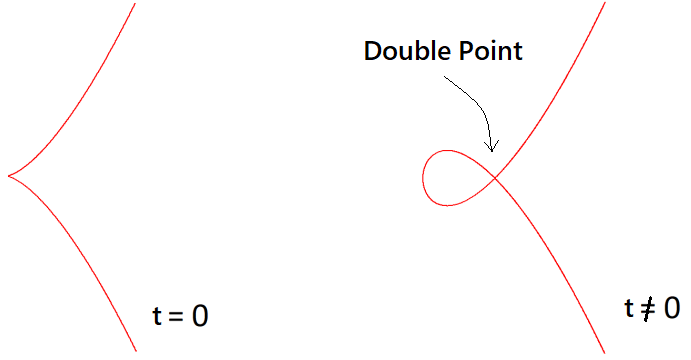} 
\caption{Double point in the stabilization of $f$ (real points)}\label{figura1}
\end{figure}

\end{example}

Multiple point spaces of a map germ from $(\mathbb{C}^n,0)$ to $(\mathbb{C}^p,0)$ with $n\leq p$ play an important role in the study of its geometry. In this section, we will deal only with double points, which will be fundamental to control the notion of good and excellent unfoldings which we will see later.  

We are interested in studying the space of double points, which is denoted by $D^2(f)$. Roughly speaking, $D^2(f)$ is the set of points $(\textbf{x},\textbf{x}^{'}) \in \mathbb{C}^n \times \mathbb{C}^n$ such that $\textbf{x} \neq \textbf{x}^{'}$ and $f(\textbf{x})=f(\textbf{x}^{'})$ or $\textbf{x}$ is a singular point of $f$. In order to see $D^2(f)$ as an analytic space, we need to present an appropriated analytic structure for it. We will follow the construction of \cite{mond87}, which is valid for holomorphic maps from $\mathbb{C}^n$ to $\mathbb{C}^p$, with $n\leq p$. 

Let us denote the diagonals of $\mathbb{C}^n \times \mathbb{C}^n$ and $\mathbb{C}^{p} \times \mathbb{C}^{p}$ by $\Delta_{\mathbb{C}^n}$ and $\Delta_{\mathbb{C}^{p}}$, respectively, and denote the sheaves of ideals defining them by $\mathcal{I}_n$ and $\mathcal{I}_{p}$, respectively. We write the points of $\mathbb{C}^n \times \mathbb{C}^n$ and $\mathbb{C}^{p} \times \mathbb{C}^{p}$ as $(\textbf{x},\textbf{x}^{'})$ and $(\textbf{X},\textbf{X}^{'})$, respectively.
Locally, 

\begin{center}
$\mathcal{I}_{n}=\langle x_1-x_1^{'},\cdots, x_n-x_n^{'} \rangle$ and $\mathcal{I}_{p}=\langle X_1-X_1^{'},\cdots, X_{p}-X_{p}^{'} \rangle$. 
\end{center}

Since the pull-back $(f \times f)^{\ast}\mathcal{I}_{p}$ is contained in $\mathcal{I}_{n}$ , there exist $\alpha_{ij}\in \mathcal{O}_{\mathbb{C}^{n} \times \mathbb{C}^n}$, such that
\[
f_{i}(\textbf{x})-f_{i}(\textbf{x}^{'})= \alpha_{i1}(\textbf{x},\textbf{x}^{'})(x_1-x_1^{'})+ \cdots + \alpha_{in}(\textbf{x},\textbf{x}^{'})(x_n-x^{'}_n), \ for \ i=1,\cdots,p.
\]

If $f(\textbf{x})=f(\textbf{x}^{'})$ and $\textbf{x} \neq \textbf{x}^{'}$, then every $n \times n$ minor of the matrix 

\begin{equation}\label{eq15}
\alpha=(\alpha_{ij})
\end{equation}

\noindent must vanish at $(x,x^{'})$. We denote by $\mathcal{R}(\alpha)$ the ideal in $\mathcal{O}_{\mathbb{C}^{p}}$ generated by the $n\times n$ minors of $\alpha$. Then we define the \textit{double point space $D^2(f)$} (as a complex space) by
\[
D^{2}(f)=V((f\times f)^{\ast}\mathcal{I}_{p}+\mathcal{R}(\alpha)).
\]

We call \textit{double point ideal} the ideal $\mathcal{I}^{2}(f)=\langle(f\times f)^{\ast}\mathcal{I}_{p}+\mathcal{R}(\alpha)\rangle$. Although the ideal $\mathcal{R}(\alpha)$ depends on the choice of the coordinate functions of $f$, in \cite{mond87} it is proved that $\mathcal{I}^{2}(f)$ does not, and so $D^{2}(f)$ is well defined. It is easy to see that the points in the underlying set of $D^{2}(f)$ are exactly the ones of type $(\textbf{x},\textbf{x}^{'})$ with $\textbf{x} \neq \textbf{x}^{'}$, $f(\textbf{x})=f(\textbf{x}^{'})$ and the ones of type $(\textbf{x},\textbf{x})$ such that $\textbf{x}$ is a singular point of $f$.

Let $f:(\mathbb{C}^n,0)\rightarrow(\mathbb{C}^{p},0)$ be a finite map germ and take a representative of $f$ defined on a small enough open neighbourhood of the origin. Denote by $I_{p}$ and $R(\alpha)$ the stalks at $0$ of $\mathcal{I}_{p}$ and $\mathcal{R}(\alpha)$. We define the \textit{the double point space of the map germ $f$} as the complex space germ 

\[
D^{2}(f)=V((f \times f)^{\ast}I_{p}+R(\alpha))
\]

\noindent We call \textit{double point ideal of the germ $f$} the ideal $I^{2}(f)=\langle (f \times f)^{\ast}I_{p}+R(\alpha)\rangle$.  Therefore 

\begin{center}
 $\mathcal{O}_{D^2(f)} \simeq  \dfrac{\mathcal{O}_{p}}{I^{2}(f)}$,
 \end{center} 

\noindent is the local ring of $D^2(f)$. In order to simplify the notation we will denote the pullback $(f \times f)^{\ast}\mathcal{I}_{p}$ by $I_{\Delta}^2(f)$.

\section{Calculating $d(f)$ using Mond's ideal of $D^2(f)$}

$ \ \ \ \ \ $ As we said in Introduction, Artin-Nagata and Gaffney presented distinct ways of calculating $d(f)$. We show in the following result that $d(f)$ can be calculated simply as the codimension of $I^2(f)$ in $\mathcal{O}_{2n}$.

\begin{proposition}\label{mainresult1} Let $f:(\mathbb{C}^n,0)\rightarrow (\mathbb{C}^{2n},0)$ be an $\mathcal{A}$-finite map germ. Then

\begin{center}
$d(f)= \dfrac{1}{2} \displaystyle \left( dim_{\mathbb{C}}\dfrac{\mathcal{O}_{2n}}{I^2(f)} \right) $.
\end{center}

\end{proposition}

\begin{proof} Since $(n,2n)$ is on the set of nice dimensions of Mather (see \cite[p. 208]{mather}), we can consider a stabilization $F:(\mathbb{C}^n \times \mathbb{C},0)\rightarrow (\mathbb{C}^{2n} \times \mathbb{C},0)$, $F=(f_t(\textbf{x}),t)$ of $f$. Note that $F$ is an analytic finite map germ from $(\mathbb{C}^{N},0)$ to $(\mathbb{C}^{2N-1},0)$, where $N=n+1$. Since $F$ is a stabilization of $f$, we have that $D^2(F)$ is one-dimensional. Therefore, by \cite[Prop. 1]{victorjuanjo}, $D^2(F)$ is Cohen-Macaulay. Consider the projection on the parameter space $p:(\mathbb{C}^n \times \mathbb{C}^n \times \mathbb{C},0)\rightarrow (\mathbb{C},0)$, defined by $p(\textbf{x},\textbf{x}^{'},t)=t$. Since $D^2(F)$ is Cohen Macaulay, the restriction $p|_{D^2(F)}:D^2(F)\rightarrow \mathbb{C}$ is flat. Note that 

\begin{equation}\label{eq2}
deg(p|_{D^2(F)})=2d(f). 
\end{equation}

By \cite[Th. D.7]{juanjomond} we have that

\begin{equation}\label{eq3}
deg(p|_{D^2(F)})=e(p^{\ast}m_{\mathcal{O}_1},\mathcal{O}_{D^2(F)}).
\end{equation}

\noindent where $m_{R}$ denotes the maximal ideal of the local ring $R$. 

Now, note that $p^{\ast}m_{\mathcal{O}_1}=\langle t \rangle \mathcal{O}_{D^2(F)}$ is a $m_{\mathcal{O}_{D^2(F)}}$-primary ideal in $\mathcal{O}_{D^2(F)}$. Since $\mathcal{O}_{D^2(F)}$ is one-dimensional, it follows that $\langle t \rangle \mathcal{O}_{D^2(F)}$ is a parameter ideal for $\mathcal{O}_{D^2(F)}$. Thus, by \cite[Th. 17.11]{matsumura},

\begin{center}
$e(p^{\ast}m_{\mathcal{O}_1},\mathcal{O}_{D^2(F)})=dim_{\mathbb{C}} \left( \dfrac{\mathcal{O}_{D^2(F)}}{\langle t \rangle \mathcal{O}_{D^2(F)}}\right) = dim_{\mathbb{C}}(\mathcal{O}_{D^2(f)})$.
\end{center}

\noindent The proof follows now by (\ref{eq2}) and (\ref{eq3}).\end{proof}

\begin{example}\label{example1} Consider the map germ $f(x,y)=(x^2,x^3-xy,y^2,y^3+xy)$. This example appears in \rm\cite[Example 3.4]{gaffney2}. \textit{It is easy to check that $f$ is finitely determined. In this example, we have that the matrix $\alpha$ is}

 \[
  \alpha=  \left[
         \begin{array}{cc}
         
x+x^{'} & 0 \\

x^2+xx^{'}+x^{'2}-y &  -x^{'}  \\

0 & y+y^{'} \\

y & y^2+yy^{'}+y^{'2}+x^{'}

\end{array}
    \right]
\]

\noindent Using {\sc Singular} \rm\cite{singular}, \textit{we find that} $2d(f)=dim_{\mathbb{C}} \ (\mathbb{C}\lbrace x,y,x^{'},y^{'}\rbrace)/I^2(f)=12$, \textit{as shown in} \rm\cite{gaffney2}.

\end{example}

\subsection{Quasihomogeneous map germs from $(\mathbb{C}^n,0)$ to $(\mathbb{C}^{2n},0)$}

$ \ \ \ $ In the section, we would like to study the invariant $d(f)$ for quasihomogeneous map germs. Thus, it is convenient to present the following definition.

\begin{definition}\label{defquasihomog} A polynomial $p(x_1,\cdots,x_n)$ is \textit{quasihomogeneous} if there are positive integers $w_1,\cdots,w_n$, with no common factor and an integer $d$ such that $p(k^{w_1}x_1,\cdots,k^{w_n}x_x)=k^dp(x_1,\cdots,x_n)$. The number $w_i$ is called the weight of the variable $x_i$ and $d$ is called the weighted degree of $p$. In this case, we say $p$ is of type $(d; w_1,\cdots,w_n)$.
\end{definition}

Definition \ref{defquasihomog} extends to polynomial map germs $f:(\mathbb{C}^n,0)\rightarrow (\mathbb{C}^p,0)$ by just requiring each coordinate function $f_i$ to be quasihomogeneous of type $(d_i; w_1,\cdots,w_n)$, for fixed weights $w_1,\cdots,w_n$. In particular, for a quasihomogeneous map germ $f:(\mathbb{C}^n,0)\rightarrow (\mathbb{C}^{2n},0)$ we say that it is quasihomogeneous of type $(d_1,\cdots,d_{2n}; w_1,\cdots,w_n)$. We are now able to state our next result. 

\begin{proposition}\label{quasihomformula} Let $f:(\mathbb{C}^n,0)\rightarrow (\mathbb{C}^{2n},0)$ be an $\mathcal{A}$-finite map germ. If $n\leq 3$ and $f$ is quasihomogeneous of type $(d_1,\cdots,d_{2n}; w_1,\cdots,w_n)$, then $d(f)$ is given by

\begin{equation}\label{eq50}
d(f)=\dfrac{1}{2w_1^2\cdots w_n^2}\left( d_1\cdots d_{2n}+ \displaystyle \sum_{\substack{\alpha+\beta=2n,\\ 0 \leq \alpha_i \leq 1\leq \beta_j \leq n+1}} (-1)^{n+\alpha+1}d_1^{\alpha_1}\cdots d_{2n}^{\alpha_{2n}}\cdot w_1^{\beta_1}\cdots w_n^{\beta_n} \right)
\end{equation}

\noindent where $\alpha=\alpha_1+\cdots + \alpha_{2n}$ and $\beta=\beta_1+\cdots \beta_{n}$.

\end{proposition}

\begin{proof} The case $n=1$ is well known and the proof is trivial. So, we will prove first the case $n=2$. Let $f:(\mathbb{C}^2,0)\rightarrow (\mathbb{C}^4,0)$ be a quasihomogeneous, $\mathcal{A}$-finite map germ and consider its $4 \times 2$ matrix $\alpha=(\alpha_{ij})$ (described in (\ref{eq15})) and the double point ideal $I^2(f)$. Consider the ring $\mathcal{O}_{10}$ with the following variables 

\begin{center}
$\mathcal{O}_{10}=\mathbb{C}\lbrace x,y,a_{11},a_{12},a_{21},a_{22},a_{31},a_{32},a_{41},a_{42}\rbrace $ 
\end{center}

\noindent and the $4 \times 2$ matrix $A=(a_{ij})$.
  
Define the ideal $I:=I_2(A)+(a_{11}x+a_{12}y, a_{21}x+a_{22}y,a_{31}x+a_{32}y,a_{41}x+a_{42}y)$ of $\mathcal{O}_{10}$, where $I_2(A)$ is the ideal generated by the minors $2 \times 2$ of $A$. Consider the analytic map germ $\varphi:(\mathbb{C}^4,0)\rightarrow (\mathbb{C}^{10},0)$, defined by

\begin{center}
$\varphi(x,y,x',y')=(x-x',y-y',\alpha_{11},\alpha_{12},\alpha_{21},\alpha_{22},\alpha_{31},\alpha_{32},\alpha_{41},\alpha_{42})$.
\end{center}

Note that $\varphi^{\ast}(I)=I^2(f)$, hence $\varphi^{-1}(V(I))=V(I^2(f))$. Using {\sc Singular} \rm\cite{singular}, one can check that $V(I)\subset (\mathbb{C}^{10},0)$ is a Cohen-Macaulay ring and $codim(V(I),\mathbb{C}^{10})=codim(V(I^2(f)),\mathbb{C}^4)=4$. Again using {\sc Singular} one can find the following minimal resolution of $\mathcal{O}_{10}/I$

\begin{equation}\label{eq16}
0 \rightarrow \mathcal{O}_{10}^4 \xrightarrow{M_4} \mathcal{O}_{10}^{15} \xrightarrow{M_3} \mathcal{O}_{10}^{20} \xrightarrow{M_2} \mathcal{O}_{10}^{10} \xrightarrow{M_1} \mathcal{O}_{10} \rightarrow \mathcal{O}_{10}/I \rightarrow 0.
\end{equation}

Hence, by \cite[Prop. C.12]{juanjomond} we have that the pullback of (\ref{eq16}) to $(\mathbb{C}^4,0)$, which is

\begin{equation}\label{eq17}
0 \rightarrow \mathcal{O}_4^4 \xrightarrow{\varphi^{\ast} M_4} \mathcal{O}_4^{15} \xrightarrow{ \varphi^{\ast} M_3} \mathcal{O}_4^{20} \xrightarrow{\varphi^{\ast} M_2} \mathcal{O}_4^{10} \xrightarrow{\varphi^{\ast} M_1} \mathcal{O}_4 \rightarrow \mathcal{O}_4/I^2(f) \rightarrow 0.
\end{equation}

\noindent is a minimal resolution of $\mathcal{O}_{I^2(f)}$ over $\mathcal{O}_4$. Since $f$ is quasihomogeneous, we can compute $dim_{\mathbb{C}}(\mathcal{O}_4/I^2(f))$ by replacing (\ref{eq17}) by the exact sequence of graded modules we obtain by replacing $\mathcal{O}_4$ by its associated graded ring $R=\mathbb{C}[x,y,x^{'},y^{'}]$, with grading $deg(x)=deg(x^{'})=w_1$ and $deg(y)=deg(y^{'})=w_2$. For since $dim_{\mathbb{C}}(\mathcal{O}_4/I^2(f))$ is finite, $I^2(f)$ must contain some power of the maximal ideal, so that $dim_{\mathbb{C}}(\mathcal{O}_4/I^2(f))$ is equal to the dimension of its associated graded module $gr(\mathcal{O}_4/I^2(f))=R/I^2(f)$ (we do not distinguish in our notation the ideal $I^2(f)$ in $\mathcal{O}_4$ and the corresponding ideal in $R$). If $R[k]$ denotes $R$ with its grading shifted by $k$ (so that $1$ has degree $-k$), then associated to (\ref{eq17}) we have the following graded resolution

\begin{center}
$0 \rightarrow \displaystyle \bigoplus_{i=1}^4 R[k_{4,i}] \xrightarrow{\varphi^{\ast} M_4} \displaystyle \bigoplus_{i=1}^{15} R[k_{3,i}] \xrightarrow{\varphi^{\ast} M_3} \displaystyle \bigoplus_{i=1}^{20} R[k_{2,i}] \xrightarrow{\varphi^{\ast} M_2} \displaystyle \bigoplus_{i=1}^{10} R[k_{1,i}] \xrightarrow{\varphi^{\ast} M_1}  \mathcal{O}_4 \rightarrow \mathcal{O}_4/I^2(f) \rightarrow 0.$
\end{center}

\noindent where 

\begin{center}
$k_{1,1}=d_3-w_2+d_4-w_1, \ k_{1,2}=d_2-w_2+d_4-w_1, \ k_{1,3}=d_1-w_2+d_4-w_1, \ k_{1,4}=d_4,$\\ 
$k_{1,5}=d_2-w_2+d_3-w_1, \ k_{1,6}=d_1-w_2+d_3-w_1, \ k_{1,7}=d_3, \ k_{1,8}=d_1-w_2+d_2-w_1, \ k_{1,9}=d_2, \ k_{1,10}=d_1$
\end{center}

\begin{center}
$k_{2,1}=d_2-w_2+d_3-w_2+d_4-w_1, \ k_{2,2}=d_2-w_1+d_3-w_2+d_4-w_1, \ k_{2,3}=d_1-w_2+d_3-w_2+d_4-w_1,$\\ $k_{2,4}=d_1-w_2+d_2-w_2+d_4-w_1, \ k_{2,5}=d_1-w_2+d_2-w_2+d_3-w_1, \ k_{2,6}=d_1-w_1+d_3-w_2+d_4-w_1,$\\ 
$k_{2,7}=d_1-w_1+d_2-w_2+d_4-w_1, \ k_{2,8}=d_1-w_1+d_2-w_2+d_3-w_1, \ k_{2,9}=d_3+d_4-w_1,$\\ 
$k_{2,10}=d_2+d_4-w_1, \ k_{2,11}=d_1+d_4-w_1, \ k_{2,12}=d_2+d_3-w_1, \ k_{2,13}=d_1+d_3-w_1,$\\ 
$k_{2,14}=d_1+d_2-w_1, \ k_{2,15}=d_3+d_4-w_2, \ k_{2,16}=d_2+d_4-w_2, \ k_{2,17}=d_1+d_4-w_2,$\\
$k_{2,18}=d_2+d_3-w_2, \ k_{2,19}=d_1+d_3-w_2, \ k_{2,20}=d_1+d_2-w_2$
\end{center}

\begin{center}
$k_{3,1}=d_1-w_2+d_2-w_2+d_3-w_2+d_4-w_1, \ k_{3,2}=d_1-w_1+d_2-w_2+d_3-w_2+d_4-w_1,$\\ 
$k_{3,3}=d_1-w_1+d_2-w_1+d_3-w_2+d_4-w_1, \ k_{3,4}=d_2-w_1+d_3+d_4-w_1, \ k_{3,5}=d_1-w_1+d_3+d_4-w_1,$\\ 
$k_{3,6}=d_1-w_1+d_2+d_4-w_1, \ k_{3,7}=d_1-w_1+d_2+d_3-w_1, \ k_{3,8}=d_2-w_2+d_3-w_2+d_4,$\\ 
$k_{3,9}=d_2-w_1+d_3-w_2+d_4, \ k_{3,10}=d_1-w_2+d_3-w_2+d_4, \ k_{3,11}=d_1-w_2+d_2-w_2+d_4,$\\ 
$k_{3,12}=d_1-w_2+d_2-w_2+d_3, \ k_{3,13}=d_1-w_1+d_3-w_2+d_4, \ k_{3,14}=d_1-w_1+d_2-w_2+d_4,$\\ 
$k_{3,15}=d_1-w_1+d_2-w_2+d_3$
\end{center}

\begin{center}
$k_{4,1}=d_1-w_1+d_2-w_1+d_3+d_4-w_1, \ k_{4,2}=d_1-w_2+d_2-w_2+d_3-w_2+d_4,$\\
 $k_{4,3}=d_1-w_1+d_2-w_2+d_3-w_2+d_4, \ k_{4,4}=d_1-w_1+d_2-w_1+d_3-w_2+d_4$.
\end{center}

For $j=1,\cdots,4$, consider the polynomial $h_j:=\sum t^{k_{j,i}}$, where for a fixed $j$ the sum runs through all defined values of $k_{j,i}$. The Poincaré series $P_M(t)$ of a graded $R$-module $M=\displaystyle \bigoplus_n M_n$ is defined by $P_M(t)=\displaystyle \sum_n (dim_{\mathbb{C}}M_n)t^n$. If $M=R[-k]$ then one calculates that $P_M(t)=t^k/(1-t^{w_1})^2(1-t^{w_2})^2$. Since the alternating sum of the Poincaré series of the modules in an exact sequence of graded modules with differentials of degree $0$, is equal to $0$, we conclude that the Poincaré series $P_M(t)$ of is equal to $(1-h_1+h_2-h_3+h_4)/(1-t^{w_1})^2(1-t^{w_2})^2$. Finally, $dim_{\mathbb{C}}(\mathcal{O}_4/I^2(f))$ is equal to $P_M(1)$ and with the aid of a computer (it is a long calculation) we can evaluating $P_M(1)$ as the limit as $t\rightarrow 1$ of $P_M(1)$ (since in this case $P_M(t)$ is in fact a polynomial). So we conclude that it is equal to

\begin{center}
$\dfrac{1}{2w_1^2w_2^2}[ d_1d_2d_3d_4-(d_1d_2+d_1d_3+d_1d_4+d_2d_3+d_2d_4+d_3d_4)w_1w_2-(d_1+d_2+d_3+d_4)(w_1^2w_2+w_1w_2^2)]$\\
$ +\dfrac{1}{2w_1w_2}(w_1^2+w_1w_2+w_2^2)$.
\end{center}

\noindent which is exactly the expression in (\ref{eq50}). Now the proof follows by Proposition \ref{mainresult1} which ensure that $dim_{\mathbb{C}}(\mathcal{O}_4/I^2(f))=2d(f)$. We note that the proof for the case where $n=3$ is analogous, the only difference is that the minimal resolution (like in (\ref{eq16})) is of length $6$.\end{proof}\\

The reader may ask if the formula in (\ref{eq50}) is true for the case where $n\geq 4$ or not. When trying to reproduce the proof for $n>3$, we will have to find a minimal resolution of length $2n$, which demands a high computational cost. Thus this technique becomes impractical for high dimensions. Several examples support the idea that in fact this formula is true for all $n\geq 1$ (the case $n=1$ is trivial). This motivates us to state the following conjecture.

\begin{conjecture} Let $f:(\mathbb{C}^n,0)\rightarrow (\mathbb{C}^{2n},0)$ be an $\mathcal{A}$-finite map germ. If $f$ is quasihomogeneous of type $(w_1,\cdots, w_n;d_1,\cdots,d_{2n})$, then $d(f)$ is given by the formula in (\ref{eq50}).
\end{conjecture}
\section{The $\epsilon$-invariant for an isolated non-normal singularity}\label{Sectionmuinvariant}

$ \ \ \ \ $ In this section we will present another way to calculate $d(f)$ (Proposition \ref{mainresult4}) in terms of the $\epsilon$-invariant for isolated non-normal singularities introduced by Greuel in \cite{greuel2}. The main result of this section can be seen as another version of Artin-Nagata formula, with a more geometric formulation and in terms of a more recent invariant. We recall the definition of the $\epsilon$-invariant defined by Greuel in \cite[Section 5]{greuel2}. 

\begin{definition} Let $X$ be an analytic variety in $\mathbb{C}^n$, with $x \in X$. We say that $x$ is an isolated non-normal point of $X$ or $(X,x)$ is an isolated non-normal singularity (INNS for short) if there exists a neighbourhood $U$ of $x$ such that $U \setminus \lbrace x \rbrace$ is a normal analytic variety.
\end{definition}

If $(X,x)$ is an INNS, then $H^0_x(\mathcal{O}_{X})$ is a finite $\mathbb{C}$-vector space, where $H^0_x$ denotes local cohomology (see \cite[p. 32]{greuel2}). Thus, following \cite{greuel2}, we define

\begin{center}
 $\epsilon(X,x)=dim_{\mathbb{C}} \ H^0_x(\mathcal{O}_{X})$,
   \end{center}   

\noindent the $\epsilon$-invariant of $(X,x)$ at the point $x$. Let $f:(\mathbb{C}^n,0)\rightarrow (\mathbb{C}^{2n},0)$ be a finite map germ. In order to simplify the notation, we will denote the pullback $(f \times f)^{\ast}\mathcal{I}_{2n}$ which we used in Section \ref{sec1} by $I_{\Delta}^2(f)$, that is,

\begin{center}
$I_{\Delta}^2(f):=\langle f_1(\textbf{x})-f_1(\textbf{x}^{'}),\cdots,f_{2n}(\textbf{x})-f_{2n}(\textbf{x}^{'}) \rangle$,
\end{center}

\noindent Set $\Delta_f=V(I_{\Delta}^2(f))$. In the following, all primary decompositions considered are minimal (or irredundant).

\begin{lemma}\label{lemma1} Suppose that $f:(\mathbb{C}^n,0)\rightarrow (\mathbb{C}^{2n},0)$ is finitely determined.  If $f$ is stable, then$I_{\Delta}^2(f)=\langle x_1-x_1^{'},\cdots,x_n-x_n^{'}  \rangle$. On the other hand, if $f$ is not stable, then $\Delta_f$ is a isolated non-normal singularity and each primary decomposition of $I_{\Delta}^2(f)$ has the form

\begin{center}
$I_{\Delta}^2(f)=\langle x_1-x_1^{'},\cdots,x_n-x_n^{'}  \rangle \cap Q$,
\end{center}

\noindent where $Q$ is a (not-unique) $m$-primary in $\mathcal{O}_{2n}$. 
\end{lemma}

\begin{proof} Note that $\sqrt{I^2_{\Delta}(f)}=\langle x_1-x_1^{'},\cdots,x_n-x_n^{'}  \rangle$. Let 

\begin{equation}\label{eq12}
I_{\Delta}^2(f)=Q_1 \cap \cdots Q_r,
\end{equation}

\noindent be a (minimal) primary decomposition of $I_{\Delta}^2(f)$, and set $P_i=\sqrt{Q_i}$. By the formula (\ref{eq10}) of Arting Nagata, we have that if $f$ is finitely determined, then the support of the $\mathcal{O}_{2n}$-module $\sqrt{I^2_{\Delta}(f)}/I^2_{\Delta}(f)$ is the maximal ideal of $\mathcal{O}_{2n}$. Since $\sqrt{I^2_{\Delta}(f)}$ is a prime ideal, this implies that $\langle x_1-x_1^{'},\cdots,x_n-x_n^{'}  \rangle$ necessarily appers in (\ref{eq12}) and the only prime ideal distinct to $\langle x_1-x_1^{'},\cdots,x_n-x_n^{'}  \rangle$  which eventually appears as the radical ideal of some of the ideals in a primary decomposition of $I_{\Delta}^2(f)$ is the maximal ideal of $\mathcal{O}_{2n}$. In other words, we have that $r=1$ or $2$, where $Q_1=\langle x_1-x_1^{'},\cdots,x_n-x_n^{'}  \rangle$ and eventually appears some $m$-primary in $\mathcal{O}_{2n}$ ideal $Q_2=Q$. 
Now, we have that $d(f)=0$ if and only if $f$ is stable. This implies that $Q$ appears in (\ref{eq12}) if and only if $f$ is not stable. Now it is clear that $\Delta^2(f) \setminus \lbrace 0 \rbrace$ is normal.\end{proof}\\

The following theorem is essentially due to Artin-Nagata \cite{artinnagata}. We just present here another point of view of Artin Nagata formula (\ref{eq10}) using the $\epsilon$ invariant of the INNS $\Delta_f$.

\begin{proposition}\label{mainresult4} If $f:(\mathbb{C}^n,0)\rightarrow (\mathbb{C}^{2n},0)$ is finitely determined, then $\epsilon(\Delta_f)=2d(f)$. Furthermore, if $I_{\Delta}^2(f)=\langle x_1-x_1^{'},\cdots,x_n-x_n^{'}  \rangle \cap Q$ is a primary decomposition of $I_{\Delta}^2(f)$, then

\begin{equation}\label{eq14}
d(f)=\dfrac{1}{2} \displaystyle \left(  dim_{\mathbb{C}}\dfrac{\mathcal{O}_{2n}}{Q} \ - \ dim_{\mathbb{C}}\dfrac{\mathcal{O}_{2n}}{\langle x_1-x_1^{'},\cdots,x_n-x_n^{'}  \rangle + Q}    \right).
\end{equation}

\end{proposition}

\begin{proof} Since $\Delta^2(f)$ is an isolated non-normal singularity we can calculate its $\epsilon$-invariant. We have that

\begin{equation}\label{eq13}
\epsilon(\Delta^2(f))=dim_{\mathbb{C}}\dfrac{\sqrt{I^2_{\Delta}(f)}}{I^2_{\Delta}(f)}=2d(f),
\end{equation}

\noindent where the first equality in (\ref{eq13}) follows by the fact that if $X=V(I)$, then $H^0_x(\mathcal{O}_{X})=\sqrt{I}/I$ (see \cite[p. 32]{greuel2}) and the second equality in (\ref{eq13}) follows by the formula (\ref{eq10}) of Arting Nagata. Now the formula in (\ref{eq14}) follows by Lemma \ref{lemma1} and \cite[Corollary 5.7]{greuel2}.\end{proof}

\begin{example} Consider the map germ $f(x,y)=(x^2,x^3-xy,y^2,y^3+xy)$ of Example \ref{example1}(b). Using {\sc Singular} we found that $I^2_{\Delta}(f)= \langle x-x^{'},y-y^{'} \rangle \cap Q$, where

\begin{center}
$Q=\langle x^{'4}, \ yx^{'}y^{'3}, \ yx^{'2}y^{'2}-x^{'2}y^{'3}, \ 2yx^{'3}-2x^{'3}y^{'}+yy^{'3}, \ y^2-y^{'2}, \ -2yx^{'}y^{'}+2xy^{'2}-yy^{'3}, \ -2yx^{'2}+2xx^{'}y^{'}+yx^{'}y^{'2}-x^{'}y^{'3}, \ 2xx^{'2}-2x^{'3}+yy^{'2}-y^{'3},  \ 2xy-2x^{'}y^{'}+yy^{'2}-y^{'3}, \ x^2-x^{'2} \rangle$
\end{center}


We have that $2d(f)=\dfrac{1}{2} \displaystyle \left(  dim_{\mathbb{C}}\dfrac{\mathbb{C}\lbrace x,y,x^{'},y^{'}\rbrace}{Q} \ - \ dim_{\mathbb{C}}\dfrac{\mathbb{C}\lbrace x,y,x^{'},y^{'}\rbrace}{\langle x-x^{'},y-y^{'}  \rangle + Q}    \right)=28-16=12$.

\end{example}

\section{Whitney equisingularity}

$ \ \ \ \ $ We are also interested in studying invariants that control the Whitney equisingularity of a deformation of $f$. We consider an unfolding $F:(\mathbb{C}^n \times \mathbb{C},0)\rightarrow (\mathbb{C}^{2n}\times \mathbb{C},0)$, $F(\textbf{x},t)=(f_t(\textbf{x}),t)$, and we assume that it is origin preserving, which means that $f_t(0)=0$ for any $t$. Gaffney in \cite[Sec. 7]{gaffney} showed that the Whitney equisingularity, hence the topological triviality, of a 1-parameter family of map germs from $(\mathbb{C}^n,0)$ to $(\mathbb{C}^p,0)$ is controlled by the multiplicities of the local polar varieties and the zero stable invariants of all the stable types which appear in the source and in the target. In order to use Gaffney's result, lets recall the definition of good and excellent unfoldings (see \cite[Def. 2.1, 6.2]{gaffney}).

\begin{definition}\label{defgood}
(a) Consider a finitely determined map germ $f: (\mathbb{C}^n,0)\rightarrow (\mathbb{C}^{2n},0)$ and let $F:(\mathbb{C}^n \times \mathbb{C},0)\rightarrow (\mathbb{C}^{2n}\times \mathbb{C},0)$, $F=(f_t(\textbf{x}),t)$ be an unfolding of $f=f_0$. We say that $F$ is a good unfolding of $f$ if:\\

(a.1) there is a representative $F : F^{-1}(V \times T) \rightarrow V \times T$ such that $f_t^{-1}(0)=\lbrace 0 \rbrace$ and

(a.2) $f_t: U_t \setminus \lbrace 0 \rbrace \rightarrow \mathbb{C}^{2n}$ is stable for any $t \in T$, where $U_t=F^{-1}(V \times \lbrace t \rbrace)$.\\

\noindent (b) We say that $F$ is excellent if it is good and $d(f_t)$ is constant.

\end{definition}

The following proposition shows how we can control the condition that $F$ is excellent. The result essentially follows by \cite[Corollary 3.5 and 3.3]{gaffney2}, we will present a proof here for completeness.

\begin{proposition}\label{propexcellent} Let $f$ and $F$ as in Definition \ref{defgood}. Then $F$ is excellent if and only if $d(f_t)$ is constant.
\end{proposition}

\begin{proof} We just need to show that if $d(f_t)$ is constant then $F$ is good. Suppose that condition (a.1) in Definition \ref{defgood} fails. That is, suppose that for every representative of $F$ (which we also call $F$), we have that $f_t^{-1}(0) \neq 0$. Since we assume that $F$ is origin preserving, we have that $f_t(0)=0$ for all $t$. Thus $\lbrace 0 \rbrace \subset f_t^{-1}(0)$, and we conclude that $f_t^{-1}(0)$ contains at least two points. Therefore $F$ is not injective, which is a contradiction by \cite[Corollary 3.5]{gaffney2}.

On the other hand, suppose that condition (a.2) in Definition \ref{defgood} fails. Then for every representative of $F$, there is an arc $(\alpha(t),t) \subset (\mathbb{C}^{2n} \times T) \setminus \lbrace 0 \rbrace \times T$ which contains the origin in its closure such that ${f_t}:(\mathbb{C}^n,S)\rightarrow (\mathbb{C}^{2n},0)$ is not stable, where $S:=f_t^{-1}(\alpha(t))$. In particular, $F$ is not an immersion in $F^{-1}(\alpha)$, which is a contradiction, again by \cite[Corollary 3.5]{gaffney2}.\end{proof}\\

We will briefly recall from \cite[Sections 6 and 7]{gaffney} the notion of Whitney equisingular unfolding $F:(\mathbb{C}^n \times \mathbb{C},0)\rightarrow (\mathbb{C}^p \times \mathbb{C},0)$ and the theorem which characterizes it in terms of the polar multiplicities of the strata defined by stable types. 

\begin{definition} Let $f:(\mathbb{C}^n,0)\rightarrow (\mathbb{C}^p,0)$ be a finitely determined map germ, with $n\leq p$.
If $F=f_t$ is a one-parameter unfolding of $f$ with parameter axis $T$, then a regular stratification $(\mathcal{X},\mathcal{X}^{'})$ of $F$ is said to be \textit{Whitney equisingular} along $T$ if $T$ is a stratum of $\mathcal{X}$ and of $\mathcal{X}^{'}$, ($\mathcal{X}$ and $\mathcal{X}^{'}$ are Whitney regular along $T$) and if any stratum $Y \in \mathcal{X}$ (respectively $\mathcal{X}^{'}$) satisfies Thom's condition over any other stratum $W \in \mathcal{X}$ (respectively $\mathcal{X}^{'}$).
\end{definition}

The polar multiplicities of the polar varieties (defined by Teissier in \cite[Ch. IV]{teissier}) of the stable types are the invariants needed to show the Whitney equisingularity of unfoldings. 

Suppose $f:(X,0)\rightarrow (S,0)$ is a flat map with smooth fibers at every point $X \setminus Sing(X)$, where $X$ is an analytic variety. Let $p:\mathbb{C}^n \rightarrow \mathbb{C}^{d-k+1}$ be a linear projection such that $ker(p)=D_{d-k+1}$, where $D_{d-k+1}$ is a linear subspace of $(\mathbb{C}^n,0)$ of dimension $k$. For $x \in X \setminus Sing(X)$, the fiber $X(f(x))$ is non-singular at $x$ contained in $\lbrace f(x) \rbrace \times \mathbb{C}^n$ and one denotes by $\pi_{x}:X(f(x))\rightarrow \mathbb{C}^{d-k+1}$ the restriction of $p$ to $X(f(x))$. Let $P_k(f,p)$ be the closure of points $x \in X \setminus Sing(X)$ such that $x \in \Sigma(\pi_x)$, one calls the closed analytic subspace $P_k(f,p)$ of $X$, the relative polar variety of codimension $k$ associated to $D_{d-k+1}$. If $f$ is the constant map, we denote this by $P_k(X)$, called absolute polar variety.

The key invariant of $P_k(f,p)$ is its polar multiplicity which we denote by $m_k(X,f)$, if $f$ is the constant map, we denote this by $m_0(P_k(X))$ or $m_k(X)$. We apply now these polar varieties and multiplicities to the strata of the stratification defined by stable types in a finitely determined map germ.

Let $f:(\mathbb{C}^n,0)\rightarrow (\mathbb{C}^p,0)$ be a finitely determined map germ of discrete stable type and let $X$ be one of the strata induced by stable types either in the source $\mathbb{C}^n$ or the target $\mathbb{C}^p$. The closure $\overline{X}$ has an analytic structure in such a way that it is reduced and o pure dimension except perhaps at $0$, and it is well behaved under deformation. Thus, if $X$ has dimension $d\geq 1$, we can consider the \textit{polar multiplicities} $m_k(\overline{X})$, for $k=0,\cdots,d-1$, where $d=dim(\overline{X})$.

Moreover, Gaffney in \cite[p. 195]{gaffney} defines a new invariant as following: take a stabilization $F:(\mathbb{C}^n \times \mathbb{C},0)\rightarrow (\mathbb{C}^p \times \mathbb{C},0)$ of $f$, and denote by $Y$ the corresponding stratum in $F$ defined by the same stable type. The restriction of the projection onto the first factor $\pi:(\overline{Y},0)\rightarrow (\mathbb{C},0)$ gives a deformation of $\overline{X}$. Then, define the \textit{d}th \textit{stable multiplicity} as $m_d(\overline{X})=m_d(\overline{Y},\pi)$. If follows that $m_d(\overline{X})$ is also an invariant of $f$, which does not depend on the stabilization $F$.

We are able now to state the main theorem of \cite[p. 206-207]{gaffney}, which characterizes Whitney equisingularity in terms of the constancy of all polar multiplicities (including the \textit{d}th stable multiplicity) of all the strata defined by stable types. For the definitions of good and excellent unfolding in the general case see \cite[Def. 2.1 and 6.2]{gaffney}.

\begin{theorem}\label{gaffneytheorem}(\rm\cite{gaffney})\textit{Suppose that $F:(\mathbb{C}^n \times \mathbb{C},0)\rightarrow (\mathbb{C}^p \times \mathbb{C},0)$ is a good unfolding of a finitely determined map germ $f$ with $(n,p)$ in the range of nice dimensions of Mather. Then $F$ is Whitney equisingular if and only if it is excellent and the polar and stable multiplicities of all the strata of the stratification gives by stable types are constant.}
\end{theorem}

Let us return now to our case $p=2n$, with $n\geq 2$. Given a finitely determined map germ $f:(\mathbb{C}^n,0)\rightarrow (\mathbb{C}^{2n},0)$, let $F:(\mathbb{C}^n \times \mathbb{C},0)\rightarrow (\mathbb{C}^{2n} \times \mathbb{C},0)$, $F(\textbf{x},t)=(f_t(\textbf{x}),t)$, be an unfolding of $f$. Thus, in order to apply Theorem \ref{gaffneytheorem}, we have to control $n+2$ invariants, namely, $d(f_t)$, $m_0(f_t(\mathbb{C}^n))$, $\cdots, \ m_n(f_t(\mathbb{C}^n))$. 

Of course, we should also control the condition that says that the unfolding has to be good, but this will be controlled by the constancy of the invariant $d(f_t)$. In the remainder of this section, we will see that the polar multiplicities $m_k(f_t(\mathbb{C}^n)$ are related, so that it is possible to reduce the number of invariants. The following proposition shows how we can calculate the polar multiplicities in terms of the Milnor number of some apropriated ICIS (isolated complete intersection singularity).

\begin{proposition}\label{propformulas} Let $f:(\mathbb{C}^n,0)\rightarrow (\mathbb{C}^{p},0)$ be a finitely determined map germ, where $p=2n$ or $2n-1$. Let $l_1,\cdots,l_{n+1}:\mathbb{C}^{2n}\rightarrow \mathbb{C}$ be generic linear forms such that if $p_k=(l_1,\cdots,l_{n-k+1})$, then 

\begin{center}
$P_k(f(\mathbb{C}^n))=\overline{\displaystyle \sum (p_k|_{f(\mathbb{C}^n) \setminus \lbrace 0 \rbrace})}$
\end{center}

\noindent and $m_k(f(\mathbb{C}^n))$, is the local degree at $0$ of $p_{k+1}|_{P_k(f(\mathbb{C}^n))}:P_k(f(\mathbb{C}^n))\rightarrow \mathbb{C}^{n-k}$, for $k=0,\cdots,n-1$.
 
Also, let $F:(\mathbb{C}^n \times \mathbb{C},0)\rightarrow (\mathbb{C}^{2n} \times \mathbb{C},0)$ be a stabilization of $f$. Assume that

\begin{center}
$P_n(F(\mathbb{C}^{n+1}),\pi)=\overline{\displaystyle \sum ((p_n,\pi)|_{F(\mathbb{C}^{n+1}) \setminus \lbrace F(D(F)) \rbrace})}$
\end{center}

Set $X_k=V(l_1 \circ f,\cdots,l_{n-k+1} \circ f)$. Then\\

\noindent (a) The stable and polar multiplicities of $f(\mathbb{C}^n)$ can be calculated as

\begin{flushleft}
(a.1) $m_0(f(\mathbb{C}^n))=1+\mu(X_1)$,

(a.2) $m_k(f(\mathbb{C}^n))=\mu(X_k)+\mu(X_{k+1})$, for $k=1,\cdots,n-1$, 

(a.3) $m_n(f(\mathbb{C}^n))=\mu(X_n)$.
\end{flushleft}

\noindent (b) We have the alternating sum of polar multiplicities

\begin{center}
$m_0(f(\mathbb{C}^n))-m_1(f(\mathbb{C}^n))+m_2(f(\mathbb{C}^n))-m_3(f(\mathbb{C}^n))+\cdots+(-1)^n m_n(f(\mathbb{C}^n))=1.$
\end{center}

\end{proposition}

\begin{proof} The proof is essentially the same presented in \cite[Th. 5]{victorjuanjo} in the case of map germs from $(\mathbb{C}^n,0)$ to $(\mathbb{C}^{2n-1},0)$.\end{proof}\\

Now we are able to reduce the number of invariants necessary to control the Whitney equisingularity of $F$.

\begin{theorem}\label{mainresult2} Let $f:(\mathbb{C}^n,0)\rightarrow (\mathbb{C}^{2n},0)$ be a finitely determined map germ and consider an unfolding $F:(\mathbb{C}^n \times \mathbb{C},0)\rightarrow (\mathbb{C}^{2n} \times \mathbb{C},0)$, $F(\textbf{x},t)=(f_t(\textbf{x}),t)$. Then the following statements are equivalent.

\begin{flushleft}
(a) $F$ is Whitney equisingular.\\
(b) $d(f_t)$ and all polar multiplicities of even index $m_0(f_t(\mathbb{C}^{n}))$, $m_2(f_t(\mathbb{C}^{n}))$, $\cdots$ are constant.\\
(c) $d(f_t)$ and all polar multiplicities of odd index $m_1(f_t(\mathbb{C}^{n}))$, $m_3(f_t(\mathbb{C}^{n}))$, $\cdots$ are constant.
\end{flushleft}

\end{theorem}

\begin{proof} The equivalence $(b) \Leftrightarrow (c)$ follows by the alternating sum of polar multiplicities (Prop. \ref{propformulas}) and the upper semicontinuity of the polar multiplicities.

(b) $\Rightarrow $ (a) Since $d(f_t)$ is constant, by Proposition \ref{propexcellent} we have that $F$ is excellent, in particular it is also good. By Proposition \ref{propformulas} (b), the constancy of all multiplicities of even index $m_0(f_t(\mathbb{C}^{n}))$, $m_2(f_t(\mathbb{C}^{n}))$, $\cdots$ implies the constancy of all multiplicities of odd index $m_1(f_t(\mathbb{C}^{n}))$, $m_3(f_t(\mathbb{C}^{n}))$, $\cdots$. So, by Theorem \ref{gaffneytheorem} we have that $F$ is Whitney equisingular.

(a) $\Rightarrow $ (b) Suppose that $F$ is Whitney equisingular and we will show that $F$ must be good. Since $F$ is Whitney equisingular, in particular it is topologically trivial, that is, there are germs of homeomorphisms:
\[
H :(\mathbb{C}^n\times \mathbb{C},0)\rightarrow (\mathbb{C}^n \times \mathbb{C},0),  \ G(x,t)=(H_{t}(\textbf{x}),t), \ H_{0}(\textbf{x})=\textbf{x}, \ H_{t}(\textbf{0})=\textbf{0}, 
\]
\[
G:(\mathbb{C}^{2n} \times \mathbb{C},0)\rightarrow (\mathbb{C}^{2n} \times \mathbb{C},0), \ G(x,t)=(G_{t}(x),t), \ G_{0}(\textbf{x})=\textbf{x}, \ G_{t}(\textbf{0})=\textbf{0},
\]
\noindent such that the diagram (\ref{eq1}) below

\begin{equation}\label{eq1}
\begin{gathered}
\xymatrix{  (\mathbb{C}^n \times \mathbb{C},0) \ar[r]^-F \ar[d]_-H  & (\mathbb{C}^{2n} \times \mathbb{C},0) \ar[d]_-G  \\
                          (\mathbb{C}^n \times \mathbb{C},0) \ar[r]^-I   &  (\mathbb{C}^{2n} \times \mathbb{C},0)}
\end{gathered}
\end{equation}
 
\noindent is commutative, where $I(\textbf{x},t)=(f(\textbf{x}),t)$ is the trivial unfolding of $f$. 

Choose a representative $F:U \times T \rightarrow V \times T$ of $F$, where $U$, $V$ and $T$ are an open neighborhoods of $0$ in $\mathbb{C}^n$, $\mathbb{C}^{2n}$ and $\mathbb{C}$, respectively. Suppose also that we choose $U$ in such a way that $f:U \rightarrow \mathbb{C}^{2n}$ is a representative such that $f^{-1}(0)=\lbrace 0 \rbrace$ and $f:U \setminus \lbrace 0 \rbrace$ is a stable immersion. Thus, we have that $I^{-1}(T)=T$. Since $H,G$ are homeomorphisms, it follows that $F^{-1}(T)=(H^{-1}\circ I^{-1} \circ G)(T)=T$, therefore we have that condition (a.1) of Definition \ref{defgood} is satisfied. 
On the other hand, the existence of an arc $\alpha(u)=(\alpha_1(u),\cdots,\alpha_{2n}(u),\alpha_{2n+1}(u))$ violating the condition $(a.2)$ of Definition \ref{defgood} must be an arc such that the restriction $f_{t}:(U \setminus 0)\rightarrow V$ is not stable in the set of points $S=\lbrace f_t^{-1}(\alpha_1(u),\cdots,\alpha_{2n}(u)) \rbrace$, for any $t\neq 0$. Since every instable point of $f_{t}:(U \setminus 0)\rightarrow V$ is a singular point of $f_t(U)$, and hence a singular point of $F(U \times T)$, we must give a one-dimensional stratum not equal to the $T$ axis. Hence, a stratification of $F$ such that $T$ is a stratum does not exist, and therefore $F$ can not be Whitney equisingular, a contradiction. 
Since $F$ must be good, the result now follows by Theorem \ref{gaffneytheorem}.
\end{proof}\\

If $f$ has corank $k < n$, then the number of invariants necessary to control the Whitney equisingularity of $F$ in Theorem \ref{mainresult2} can be reduced, as we will see in the next result. In the following, 

\begin{center}
$odd(x)= 2 \left\lceil\dfrac{x+1}{2}\right\rceil -1 $, $ \ \ \ $ $\displaystyle \left(respect. \ even(x)= 2 \left\lfloor\dfrac{x-1}{2}\right\rfloor \right) $, 
\end{center}

\noindent is the function that takes as input a real number $x$ and gives as output the greatest odd integer (respect. greatest even integer) less than or equal to $x$. Clearly we have that if $k$ is an even integer, then $odd(k)=k-1$ and $even(k)=k$. On the other hand, if $k$ is an odd integer, then $odd(k)=k$ and $even(k)=k-1$.

\begin{theorem}\label{mainresult3} Let $f:(\mathbb{C}^n,0)\rightarrow (\mathbb{C}^{2n},0)$ be a finitely determined map germ of corank $k$ and consider an unfolding $F:(\mathbb{C}^n \times \mathbb{C},0)\rightarrow (\mathbb{C}^{2n} \times \mathbb{C},0)$, $F(\textbf{x},t)=(f_t(\textbf{x}),t)$. Then the following statements are equivalent.

\begin{flushleft}
(a) $F$ is Whitney equisingular.\\
(b) $d(f_t)$ and the polar multiplicities of even index $m_0(f_t(\mathbb{C}^{n}))$, $m_2(f_t(\mathbb{C}^{n}))$, $\cdots$, $m_{even(k)}(f_t(\mathbb{C}^n))$ are constant.\\
(c) $d(f_t)$ and all polar multiplicities of odd index $m_1(f_t(\mathbb{C}^{n}))$, $m_3(f_t(\mathbb{C}^{n}))$, $\cdots, m_{odd(k)}(f_t(\mathbb{C}^{n}))$ are constant.
\end{flushleft}

\end{theorem}

\begin{proof} Since $f$ has corank $k\geq 1$, after a suitable change of coordinates we can write $f$ in the following form:

\begin{equation}\label{eq6}
f(x_1,\cdots,x_n)=(x_1,\cdots,x_{n-k},f_{(n-k)+1}(\textbf{x}),\cdots,f_{2n}(\textbf{x})).
\end{equation}

Let $l_1,\cdots,l_{n+1}:\mathbb{C}^{2n}\rightarrow \mathbb{C}$ be generic linear forms as in Proposition \ref{propformulas}. With the same notation used in Proposition \ref{propformulas}, we have that $X_i=V(l_1 \circ f, \cdots, l_{n-i+1} \circ f)$. So, for each $i\geq k+1$ we have the following isomorphisms between local rings

\begin{center}

$\mathcal{O}_{X_i} \simeq \dfrac{\mathbb{C}\lbrace x_1,\cdots,x_n \rbrace}{\langle l_1 \circ f, \cdots, l_{n-i+1} \circ f \rangle} \simeq \mathbb{C}\lbrace x_{(n-k)+1},\cdots,x_n \rbrace$,
\end{center}

\noindent where the second isomorphism follows by (\ref{eq6}). 

Hence, $(X_i,0)$ is smooth for $i=k+1,\cdots,n$, we have that $\mu(X_i)=0$, for $i=k+1,\cdots,n$. By Proposition \ref{propformulas}, we have that $m_i(f_t(\mathbb{C}^n))=0$ for $i=k+1,\cdots,n$. So, we just need to consider the polar multiplicities $m_0(f_t(\mathbb{C}^n)),\cdots,m_k(f_t(\mathbb{C}^n))$. Now, the result follows by Theorem \ref{mainresult2}.\end{proof}\\

If $f$ has corank $1$ or $2$, then we need only two invariants to control the Whitney equisingularity of $F$.

\begin{corollary} \label{corollary1} Let $f:(\mathbb{C}^n,0)\rightarrow (\mathbb{C}^{2n},0)$ be a finitely determined map germ of corank $k \leq 2$ and consider an unfolding $F:(\mathbb{C}^n \times \mathbb{C},0)\rightarrow (\mathbb{C}^{2n} \times \mathbb{C},0)$, $F(\textbf{x},t)=(f_t(\textbf{x}),t)$. Then the following statements are equivalent.

\begin{flushleft}
(a) $F$ is Whitney equisingular.\\
(b) $d(f_t)$ and $m_1(f_t(\mathbb{C}^{n}))$ are constant.
\end{flushleft}

\end{corollary}

As an application of Proposition \ref{propformulas} we will show how we can calculate the Euler obstruction of the image of a map germ from $(\mathbb{C}^n,0)$ to $(\mathbb{C}^{2n},0)$. The local Euler obstruction for nonsingular varieties, introduced by R. MacPherson in \cite{macpherson}, in a purely obstructional way is a topological invariant that is also associated to the polar multiplicities. The local Euler obstruction plays an important role in his affirmative response to a conjecture of Deligne and Grothendieck on the existence of Chern class for singular complex algebraic varieties (see \cite{kennedy} and \cite{macpherson}.) Definitions equivalent to the one given by MacPherson's have been given by several authors. We cite for instance the one given in \cite{brasselet2}, see also \cite[pp. 38-42]{brasselet2}.

 Lê and Teissier in \cite{leteissier} showed that the local Euler obstruction is an alternating sum of the multiplicities of the local polar varieties. This is an important formula for computing the local Euler obstruction, and we use it in this paper as an alternative to the definition of the Euler obstruction of an analytic space $X$ at $0$, denoted by $Eu_0(X)$.

\begin{theorem}(\rm\cite[\textit{Cor.} \rm 5.1.4]{leteissier})\label{eulerobstrctionth} \textit{Let $X$ be an analytic space of dimension $d$ reduced at $0 \in \mathbb{C}^{n+1}$. Then}

\begin{center}
$Eu_0(X)= \displaystyle \sum_{i=0}^{d-1}(-1)^{d-i-1}m_{d-i-1}(X)$,
\end{center}

\noindent \textit{where $m_i(X)$ denote the polar multiplicity of the polar varieties $P_i(X)$}.
\end{theorem}

\begin{corollary}\label{mainresult5} Let $f: (\mathbb{C}^n,0)\rightarrow (\mathbb{C}^{p},0)$ be a finitely determined map germ. Consider the following cases where $p=2n$, $p=2n-1$, with $n\geq 2$ or $(n,p)=(3,4)$. Let $H$ be a generic hyperplane in $\mathbb{C}^p$. Then,

\begin{center}
$Eu_0(f(\mathbb{C}^n))=1-(-1)^n\mu(f^{-1}(H \cap f(\mathbb{C}^n))$. 
\end{center}

\noindent In particular, if $f$ has corank $k\leq n-1$, then $Eu_0(f(\mathbb{C}^n))=1$.

\end{corollary}

\begin{proof} Let $l:\mathbb{C}^p \rightarrow \mathbb{C}$ be a linear form, $l(x_1,\cdots,x_p)=a_1x_1+\cdots+a_px_p$, with $a_1,\cdots,a_p$ generic constants. The hyperplane $H=V(a_1x_1+\cdots+a_px_p)$ is generic. Set $X_n=V(f^{-1}(H \cap f(\mathbb{C}^n))$.  By Proposition \ref{propformulas} (see \cite[Th. 3.5]{victor} for the case where $(n,p)=(3,4)$) we have that

\begin{center}
$1-(-1)^n\mu(f^{-1}(H \cap f(\mathbb{C}^n))=1-(-1)^n m_n(f(\mathbb{C}^n))= \displaystyle \sum_{i=0}^{d-1}(-1)^{d-i-1}m_{d-i-1}(f(\mathbb{C}^n))=Eu_0(f(\mathbb{C}^n))$.
\end{center}

Now, if $f$ has corank $k\leq n-1$, then $X_n$ is smooth. Hence, $\mu(m_n(f(\mathbb{C}^n))=0$.\end{proof}

\begin{example} Consider the map germ $f(x,y)=(x^2,y^2,x^3+y^3+xy)$. We have that 

\begin{center}
$m_2(f(\mathbb{C}^2))=\mu(ax^2+by^2+c(x^3+y^3+xy))=1$, 
\end{center}

\noindent where $a,b,c$ are generic constants. So, $Eu_0(f(\mathbb{C}^2))=0$.

\end{example}

\begin{flushleft}
\textit{Acknowlegments:} Nuño-Ballesteros, J.J. is partially supported by Grant PID2021-124577NB-I00 funded by MCIN/AEI/10.13039/501100011033 and by ``ERDF A way of making Europe''. Silva, O.N. acknowledges support by FAPESP, grant 2020/10888-2 related to the Thematic project grant 2019/07316-0. This work was done during a stay of the second named author in the Universidade Federal de São Carlos (Brazil). He is grateful especially to the Departament of Mathematics by their hospitality and support.
\end{flushleft}

\begin{flushleft}
$\bullet$ Nuño-Ballesteros, J.J.\\
\textit{juan.nuno@uv.es}\\
Departament de Geometria i Topologia, Universitat de València, Campus de Burjassot, 46100 Burjassot, Spain and Universidade Federal da Paraíba, 58.051-900, João Pessoa, PB, Brazil.\\

$ \ \ $\\

$\bullet$ Silva, O.N.\\
\textit{otoniel.silva@academico.ufpb.br}\\
Universidade Federal da Paraíba, 58.051-900, João Pessoa, PB, Brazil.\\

$ \ \ $\\

$\bullet$ Tomazella, J.N.\\
\textit{jntomazella@ufscar.br}\\
Universidade Federal de São Carlos, Caixa Postal 676, 13560-905 São Carlos, SP, Brazil.

\end{flushleft}

\end{document}